\documentclass[reqno,12pt]{amsart}
\usepackage[a4paper,hmargin={2.7cm,2.7cm},vmargin={2.5cm,2.5cm}]{geometry}
\usepackage{amsthm,amssymb,amsmath,amscd}
\usepackage{amsthm}
\usepackage{pgf, tikz}
\usepackage{graphicx,subcaption}
\usepackage{verbatim}
\usepackage{mathrsfs}
\usepackage{amsrefs}
\usetikzlibrary{calc,decorations.pathmorphing}

\usepackage{enumitem}
\def\rmlabel{\upshape({\itshape \roman*\,})}

\def\alabel{\upshape({\itshape \alph*\,})}
\def\Alabel{\upshape({\itshape \Alph*\,})}

\usepackage{xcolor}
\usepackage{hyperref}
\hypersetup{
	colorlinks,
	linkcolor={red!60!black},
	citecolor={green!60!black},
	urlcolor={blue!60!black}
}

\makeatletter
\def\moverlay{\mathpalette\mov@rlay}
\def\mov@rlay#1#2{\leavevmode\vtop{   \baselineskip\z@skip \lineskiplimit-\maxdimen
		\ialign{\hfil$\m@th#1##$\hfil\cr#2\crcr}}}
\newcommand{\charfusion}[3][\mathord]{
	#1{\ifx#1\mathop\vphantom{#2}\fi
		\mathpalette\mov@rlay{#2\cr#3}
	}
	\ifx#1\mathop\expandafter\displaylimits\fi}
\makeatother

\usepackage{todonotes}
\newcommand{\td}[1]{\relax}

\newcommand{\blu}[1]{\textcolor{blue!80!black}{#1}}
\newcommand{\red}[1]{\textcolor{red!80!black}{#1}}

\newcommand{\df}[1]{{\bf\emph{#1}}}

\theoremstyle{definition}
\newtheorem{definition}{Definition}[section]

\theoremstyle{plain}
\newtheorem{theorem}[definition]{{Theorem}}
{Lemma}
{Theorem}
\newtheorem{lemma}[definition]{{Lemma}}

\newtheorem{fact}[definition]{Fact}
\newtheorem{claim}[definition]{Claim}
{Conjecture}
\newtheorem{defi}[definition]
{{Definition}}

\theoremstyle{remark}

\numberwithin{equation}{section}
\numberwithin{figure}{section}

\newcommand{\eps}{\varepsilon}
\newcommand{\rs}{r^*}
\newcommand{\kepp}{K_{s}^{\eps}}
\newcommand{\Gepp}{\mathbf{G}_{s}^{\eps}}
\newcommand{\tG}{\tilde G_n}

\newcommand{\bw}{\mathbf{w}}

\title[Ramsey number for cycles]
{On the restricted size Ramsey number\\ for a pair of cycles}

\author{Tomasz \L{}uczak}
\address{Adam Mickiewicz University\\
Faculty of Mathematics and Computer Science\\
Umultowska 87,
 61-614 Pozna\'n, Poland}
\email{\tt tomasz@amu.edu.pl}
\author{Joanna Polcyn}
\address{Adam Mickiewicz University\\
Faculty of Mathematics and Computer Science\\
Umultowska 87,
 61-614 Pozna\'n, Poland}
\email{\tt joaska@amu.edu.pl}
\thanks{
The last author  was supported in part by a grant from IPM (No.1401050039) 
and by insf (N0. 99013670)
}
\author{Zahra Rahimi}
\address{
Institute for Research in Fundamental Sciences (IPM)\\
School of Mathematics\\
 P.O. Box: 19395-
5746\\
 Tehran, Iran 
}
\email{\tt zahra.rahimi@alumni.iut.ac.ir}



\begin{document}

\maketitle

\begin{abstract}
For graphs $H_1,H_2$ by $r^*(H_1,H_2)$ we denote the minimum number of edges in 
a graph $G$ on $r(H_1,H_2)$ vertices such that $G\to (H_1,H_2)$. We show that for 
each pair of natural numbers $k,n$, $k\le n$, where $k$ is odd and $n$ is large enough, we 
have 
$$\rs(C_n,C_k)=\lceil (n+1)(2n-1)/2\rceil \,.$$
\end{abstract}

\section{Introduction}

In the paper we use the standard arrow notation and write $G\to (H_1,H_2, \dots, H_r)$ if any coloring of the edges of a graph
$G$ with $r$ colors leads to a copy of $H_i$ in the $i$th 
color for some $i=1,2, \dots, r$. Let us recall that the Ramsey number $r(H_1,H_2, \dots, H_r)$ is the minimum number of vertices 
in graph $G$ such that $G\to (H_1,H_2, \dots, H_r)$, while the size Ramsey number $\hat r(H_1,H_2, \dots, H_r)$  is the minimum number of edges in a graph~$G$ with such a property. In general Ramsey numbers are hard to find, yet it have been computed, or estimated, for some simple families of graphs. In particular, $r(H_1,H_2)$ is known for $H_1,H_2$ which are either paths or cycles (see Gerencs{\'e}r, Gy{\'a}rf{\'a}s~\cite{GG}, Rosta~\cite{R}, Faudree, Schelp~\cite{FS}, or Radziszowski~\cite{Ra} for a survey of results on Ramsey Theory). The size Ramsey number 
for paths and cycles are not so well understood. In the fundamental paper on this subject
Beck \cite{B}  proved that $\hat r(P_n,P_n)=O(n)$.
The exact value of the hidden constant is not known and finding it seems to be hard. The best estimates we have got so far are those by Bal, DeBiasio~\cite{BD}, and Dudek, Pra{\l}at~\cite{DP}  
$$ (3.75-o(1))n\le \hat r(P_n,P_n) \le 74n\,.$$

The size Ramsey number for cycles was studied by  Haxell, Kohayakawa, and {\L}uczak~\cite{HKL} who proved that 
$\hat r(C_n,C_n)\le An$ 
for some huge  constant $A$. The value of the constant~$A$ has been substantially improved  by Javadi, Khoeini,  Omidi, and  Pokrovskiy \cite{JKOP}, and then by Javadi and Miralaei \cite{JM};
both of the above papers  also considered the case of more than one color. 
However, finding the exact value of the limit $\hat r(C_n, C_k)/n$ (or even prove that it exists) seems to be out of  reach right now. 

In this paper we study yet another version of Ramsey number. By the restricted size Ramsey number  {$\rs(H_1,H_2, \dots, H_r)$} we denote the minimum number of edges in a graph $G$ on $r(H_1,H_2,\dots, H_r)$ vertices such that $G\to (H_1,H_2, \dots, H_r)$. Although this notion has been around for quite some time (see Faudree and Schelp~\cite{FS}) there are rather few results on the restricted Ramsey number in the literature.
The goal of this note is to find  the exact value of
$\rs(C_n, C_{k})$ where $k$ is odd and $n$ is large. Note that for such choice of $k$ and $n$, $r(C_n, C_k)=2n-1$ and one of the colorings used to show the lower bound is when we divide $K_{2n-2}$ into two equal parts and color all edges between them in the second  color.

\begin{theorem}\label{thm:main}
There exists a constant $n_0$ such that for every $n\ge n_0$ and any odd $k$, with 
$3\le k \le n$, we have
$$\rs(C_n,C_k)=\lceil (n+1)(2n-1)/2\rceil \,.$$
\end{theorem}

The structure of the paper is the following. In the next section we derive Theorem~\ref{thm:main} from two Lemmata~\ref{l:RL} and~\ref{l:D} on colorings of a pseudorandom graph $\tG$. A precise definition of $\tG$, and the way one can construct it, is given in Section~\ref{sec:G}, while the proofs of Lemmata \ref{l:RL} and \ref{l:D} can be found in Sections~\ref{sec:RL} and~\ref{sec:D} respectively. 

We conclude this part of the paper with some remarks and comments. 
Note that calculating $r^*(C_n,C_n)$,  in a way, supplements results related to Shelp's problem who asked about the minimum $a$ so that every  graph $G$ on $N=r(C_n,C_n)$ vertices and the minimum degree $aN$ has the property $G\to (C_n,C_n)$ (see Benevides {\it et al.}~\cite{Ben} for the solution of Shelp's problem, and Łuczak and Rahimi~\cite{LZ1, LZ2} for the case of three colors). 

Finding $r^*(H_\ell, H_\ell)$ can be also viewed as a special case of the following question: given a graph $H_\ell$ on $n$ vertices, and a function $k=k(\ell)$, what is the minimum number of edges in  a graph $G$ on  $r(H_\ell,H_\ell)+k$ vertices such that $G\to (H_\ell,H_\ell)$? Although in this formulation this problem looks slightly artificial, a moment of reflection reveals that it could be (and usually is) related to the following question on random graphs: what is the largest $\ell$ 
so that  with probability $1-o(1)$ each coloring of the edges of the random graph $G(n,p)$  with two colors  leads to a monochromatic copy of $H_\ell$? As far as we know this problem was first addressed and solved for paths  by Letzter~\cite{Let} when $p=p(n)=\Theta(1/n)$, whereas for cycles and all range of $p=p(n)$ it was treated in the recent paper of Ara\'ujo {\it et al.}~\cite{A}.

Finally, one can ask about the value of $r^*(C_n,C_k)$ when $k\le n$ is an even number. This question seems to be  particularly interesting when $k$ is much smaller than $n$. For instance, in the simplest case when $k=4$, we have  $r(C_n, C_4)=n+1$ and so one can  expect that $r^*(C_n,C_4)=n^{1/2+o(1)}$, but at this moment the lower and  upper bounds for $r^*(C_n,C_4)$ we can prove are both far from the conjectured $n^{1/2+o(1)}$.

\section{Proof of Theorem~\ref{thm:main}}

Here we describe how to show Theorem~\ref{thm:main} using two Lemmata~\ref{l:RL} and~\ref{l:D} proofs of which we postpone until the following sections of the article. As we shortly see the lower bound for $r^*(C_n,C_k)$  is quite simple to obtain, 
so the main part of the proof is to find a graph $\tG$ with $r(C_n,C_k)$ vertices and $r^*(C_n,C_k)$ edges, 
and verify that any coloring of edges of $\tG$ leads to either $C_n$ in the first color, or $C_k$ in the other color. 

For $\tG$ we take a pseudorandom $(n+1)$-regular (or almost $(n+1)$-regular, if $n$ is even) graph on $2n-1$ vertices which we call $n$-fit; in the following  Section~\ref{sec:G} we give a precise definition of $n$-fit graphs and describe how to construct them. Our argument  consists of two parts. Firstly, we use the Regularity Lemma to verify that any coloring of $\tG$ without monochromatic $C_n$ and $C_k$ in the appropriate colors resembles the extremal coloring of $K_{2n-2}$ which avoids such $C_n$ and $C_k$, i.e. it contains 
a large induced bipartite subgraph in one of the colors. Our result,  proved in Section~\ref{sec:RL}, can be stated as follows. 

\begin{lemma}\label{l:RL}
There exists $n_1$ such that for every $n\ge n_1$ the following holds. If $k\le n$ is odd and the edges of an $n$-fit graph $\tG$ are colored with two colors so that there are no copies of $C_n$ in the first color and  no copies of $C_k$ in the second 
color, then there exists two disjoint subsets of vertices of $\tG$, $V'$ and $V''$, $|V'|,|V''|\ge 0.99n$, such that 
all edges between $V'$ and $V''$ are of the same color.
\end{lemma}

Then we supplement the above with the following result proved in Section~\ref{sec:D}.

\begin{lemma}\label{l:D}
There exists $n_2$ such that for every $n\ge n_2$ the following holds. 
Let us suppose that the edges of an $n$-fit graph $\tG$ are colored with two colors so that there exists two disjoint subsets of vertices $\tG$, $V'$ and $V''$, $|V'|,|V''|\ge 0.99n$, such that  all edges between $V'$ and $V''$ are of the same color.
Then the subgraph induced by one of the colors contains a copy of cycle $C_\ell$ for each $3\le \ell\le n$.
\end{lemma}

Now we are ready to complete the proof of Theorem~\ref{thm:main}.

\begin{proof}[Proof of Theorem~\ref{thm:main}] Clearly, for $n\ge \max\{n_1,n_2\}$ we must have $\tG\to (C_n,C_k)$
since otherwise  Lemmata~\ref{l:RL} and~\ref{l:D} lead to a contradiction. 
Hence $\rs(C_n,C_k)\le \lceil (n+1)(2n-1)/2\rceil$.

Moreover, each graph $G=(V,E)$ such 
that $|V|=r(C_n,C_k)=2n-1$ and  $|E|<\lceil (n+1)(2n-1)/2\rceil$,  contains a vertex $v$ of degree at most~$n$. Let $V'\subseteq V$ be a set of $n-1$ vertices which contains all, except at most one, neighbors of $v$, and $V''=V\setminus V'$.
Now  color all edges between $V'$ and $V''$ with the second color and all other edges by the first color.  
Then, there are no $C_n$ with edges colored in the first color and no odd cycles with edges colored in the second color.
\end{proof}

\section{$n$-fit graphs}\label{sec:G}

Our proofs of  Lemmata~\ref{l:RL} and~\ref{l:D} rely
on the fact that in the graph  $\tG$ we are to color the edges are `uniformly' distributed. The following definition makes it precise. Here, for  two sets of vertices $S$ and  $T$
of a graph $G=(V,E)$, we put
$$e(S,T)=|\{(v,w): v\in S, w\in T, \{v,w\}\in E\}|,$$ 
i.e. $e(S,T)$ denote the number of edges with one end in $S$, the second in $T$, where all edges contained in $S\cap T$ are counted twice. Moreover, by $N_G(v)=N(v)$ we denote all neighbors of the vertex $v$ in $G$.
\begin{defi}
A graph $G=(V,E)$ is called  {\bf $\mathbf{n}$-fit} if the following holds:
	\begin{enumerate}[label=\Alabel]
		\item\label{it:i} $|V|=2n-1$ and $|E|=  \lceil (n+1)(2n-1)/2\rceil    $;
		\item\label{it:ii} the minimum degree of $G$ is $n+1$ (i.e. $G$ is $(n+1)$-regular if $n$ is odd and 
			it has $2n-2$ vertices of degree $n+1$ and one of degree $n+2$ if $n$ is even);
		\item\label{it:iii} for all $v,w\in V$, $v\neq w$, we have
			$$\big| |N(v)\cap N(w)|- n/2\big|\le n^{0.7};$$
		\item\label{it:iv} for all $S,T\subseteq V$ we have 
			$$\big| e(S,T)- |S||T|/2\big|\le n^{1.7}.$$
	\end{enumerate}
\end{defi}

The main result of this section states that for every $n$ large enough one can find an  $n$-fit graph.

\begin{lemma}\label{l:fit}
	There exists $n_0$ such that   $n$-fit graphs exist for each $n\ge n_0$.
\end{lemma}  
\begin{proof}
Although it is expected that a random graph with a degree sequence as described in \ref{it:i} and \ref{it:ii} fulfills also conditions
\ref{it:iii} and \ref{it:iv}, instead  of employing somewhat complicated formula for the number of graphs with a given degree sequence found by Liebenau and Wierman \cite{LW}, we use the standard binomial model of random graph, which, due to the
independence of edges, is much easier to deal with. 

Consider the random graph $\mathcal{ G}_{2n-1}$ which is chosen at random from all graphs with $2n-1$ vertices, 
or, equivalently, a graph with vertex set  $[2n-1]=\{1,2,\dots,2n-1\}$ in which each edge appears independently with probability $1/2$. It is easy to verify
using Chernoff bounds that for $n$ large enough with positive probability (in fact with probability $1-o(1)$) the graph $\mathcal {G}_{2n-1}$ has
the following properties:
	\begin{enumerate}[label=\alabel]
		\item\label{it:a} for every vertex $v$ we have
				$$\big|\deg(v)-n\big|\le n^{0.6};$$
		\item\label{it:b} for every pair of vertices $v$, $w$
				$$\big| |N(v)\cap N(w)|-n/2  \big|\le n^{0.6}; $$
		\item\label{it:c} for every two subsets of vertices $S,T$
				$$\big|e(S,T)-|S||T|/2\big|\le n^{1.6}.$$
	\end{enumerate}
 Thus, let $\hat G=([2n-1],\hat E)$ be a graph for which \ref{it:a}, \ref{it:b}, and \ref{it:c} hold. Clearly, 
 it has also properties \ref{it:iii} and \ref{it:iv} with  a large margin. In order to adjust its degree sequence 
 we modify slightly the edge set  of $\hat G$ and 
 show that there exists sets $E_1,E_2\subseteq [2n-1]^{(2)}$ such that the graph $\tG=([2n-1], (\hat E\setminus E_1)\cup E_2)$ has all vertices 
 of degree $n+1$, except, perhaps one of degree $n+2$, and furthermore the maximum degree of the graph
 $G_1=([2n-1], E_1\cup  E_2)$ is smaller than $136 n^{0.6}$, so \ref{it:iii} and \ref{it:iv} hold and the graph $\tG$ 
 is $n$-fit.
 
 Let us first delete from $\hat G=([2n-1],\hat E)$ some edges so that  the resulting graph  has maximum degree $n+1$. 
 We do it recursively examining all vertices of $\hat G$ one by one. For each vertex $v$ we need to 
 delete not more than  $\deg(v)-(n+1)$ surplus edges incident to $v$ which connect $v$ to those of its neighbors whose degrees changed the least during the procedure. Note that 
 from \ref{it:a} it follows that in the whole process we remove  $|E'_1|\le 2n^{1.6}$ edges. We want to argue, 
 by a direct recursive argument, that a graph 
 $([2n-1],E'_1)$ has maximum degree at most $7n^{0.6} $. To this end note that because of $|E'_1|\le 2n^{1.6}$, in each stage at most $2n/3$ vertices has more than $6n^{0.6}$ incident edges deleted. 
 Thus, since by \ref{it:a} and the recursive assumption in each step of the process the vertex $v$ has 
 at least $n-n^{0.6}-7n^{0.6}>2n/3+n^{0.6}$ 
 neighbors in $\hat G$,  we never delete any edge which join $v$ to its neighbor who already lost at least $6n^{0.6}$ edges incident to it. Because we remove at most $n^{0.6}$ edges incident to $v$, its degree in   $([2n-1],E'_1)$ is at most $7n^{0.6}$.
 
Now we need to add edges to $\hat G=([2n-1],\hat E\setminus E'_1)$ so that the resulting graph has the correct degree sequence. Note that by \ref{it:a} and the estimate for $E'_1$ we need to add to it  $|E''|\le 4n^{1.6}$ edges. In fact we add to it 
 $E_2$,  $E_2=2|E''|$, edges and delete another $E''_1$, $|E''_1|=|E''|$, edges recursively in the following way. 
Let us suppose that two vertices, $v'$ and $v''$ have degree smaller than required. We find two other vertices, 
$w'$ and $w''$, such that $w'$ is not incident to $v'$, $w''$ is not incident to $v''$, but $w'$ and $w''$ are connected by an edge.  Then we delete from the graph an edge $w'w''$ (and put it in the set $E''_1$) and add to the graph (and to the set 
$E_2$) edges $v'w'$ and $v''w''$. In order to assure that the graph  $\hat G=([2n-1], E''_1\cup E_2)$ has small maximum degree for each pair of vertices $v'$ and $v''$ we select $w'$ and $w''$ in the following way. We first choose 
a set $W'$ of $n/10$  non-neighbors of $v'$ which are neighbors of $v''$ and which have smallest degree in currently 
generated part of $\hat G=([2n-1], E''_1\cup E_2)$. In a similar way $W''\subseteq V\setminus W'$ is the set of 
$n/10$ vertices of smallest degree, counted in currently 
generated part of $\hat G=([2n-1], E''_1\cup E_2)$,  in the set $N(v')\setminus N(v'')$. 
		\begin{figure}[ht]
		\centering		
		\begin{tikzpicture}[scale=.7]			
			\coordinate (v1) at (-1,1.5);
			\coordinate (v2) at (8,1.5);
			\coordinate (v5) at (2,1.7);
			\coordinate (v6) at (5,1.7);
			\coordinate (v3) at (2,1.5);
			\coordinate (v4) at (5,1.5);			
			\fill [black!10] (v1) -- (1.9, 2.49) -- (1.9,.51) --cycle;
			\fill [black!10] (v2) -- (5.1, 2.49) -- (5.1,.51) --cycle;			
			\draw [dashed] (v1) -- (1.9, 2.49) -- (1.9,.51) --cycle;
			\draw [dashed] (v2) -- (5.1, 2.49) -- (5.1,.51) --cycle;			
			\draw [thick] (v3) ellipse (.7cm and 1cm);
			\draw [thick] (v4) ellipse (.7cm and 1cm);			
			\fill [black!10] (v3) ellipse (.7cm and 1cm);
			\fill [black!10] (v4) ellipse (.7cm and 1cm);			
			\draw [dashed] (v1) -- (v5);
			\draw [dashed] (v2) -- (v6);
			\draw [thick] (v5)--(v6);			
			\foreach \i in {1,2,5,6} \fill (v\i) circle (2pt);			
			\node [left] at (v1) {\footnotesize  $v'$};
			\node [right] at (v2) { {\footnotesize $v''$}};
			\node [above] at (v5) {\footnotesize $w'$};
			\node [above] at (v6) {\footnotesize $w''$};
			\node at (2.1, .8) {$W'$};
			\node at (5.1, .8) {$W''$};			
		\end{tikzpicture}		
		\caption{The switching procedure.}
		\label{fig:ww}					
	\end{figure}
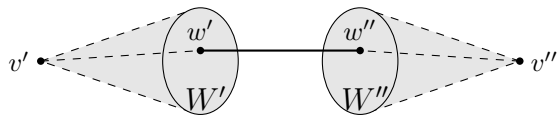
Note that because of $|E''_1\cup E_2|\le 12 n^{1.6}$, 
the graph $\hat G$ contains at most $n/5$ vertices of degree larger than $120 n^{0.6}$ and since both sets 
$N(v')\setminus N(v'')$ and $N(v'')\setminus N(v')$ have at least $n/2-O(n^{0.6})\ge n/3$ vertices in the above procedure 
none of the sets $W'$ and $W''$ contains a vertex whose  degree has been changed more than 
$120 n^{0.6}$ times. 
Moreover, by \ref{it:c}, between sets 
$W'$ and $W''$ there exist at least $|W'||W''|/2- O( n^{1.6})=\Theta(n^2)$ edges, so we can always choose two adjacent vertices $w'$ and $w''$ 
from these two sets. Hence, since by \ref{it:a} and the first part of the proof all vertices of  $\hat G=([2n-1],\hat E\setminus E'_1)$ have degree at least $n-8n^{0.6}$, the maximum degree of $\hat G=([2n-1], E''_1\cup E_2)$ is bounded from above by 
$128n^{0.6}+1$.

Thus, we have shown that the maximum degree of $([2n-1], E'_1\cup E''_1\cup E_2)$ is bounded from above by $136n^{0.6}$. Hence from \ref{it:b} it follows that for any two vertices $v$ and $w$ in $\tG=([2n-1], (\hat E\setminus (E'_1 \cup E''_1))\cup  E_2)$   we have 
$$\big| |N(v)\cap N(w)|-n/2  \big|\le n^{0.6}+272 n^{0.6}\le n^{0.7}, $$
and, by \ref{it:c},  for any two sets $S$ and $T$  of vertices of $\tG$
 $$\big|e(S,T)-|S||T|/2\big|\le n^{1.6}+136 n^{0.6} \max\{|S|,|T|\}\le n^{1.7} .\qed$$
 \renewcommand{\qed}{ }
 \end{proof}

\section{Proof of  Lemma~\ref{l:RL}}\label{sec:RL}

The main tool in the proof of Lemma~\ref{l:RL} is the Regularity Lemma so we start with recalling some definitions and 
basic facts (for a survey of results concerning the Regularity Lemma see Koml\'{o}s and Simonovits~\cite{KS}).

Let $G=(V,E)$ be a graph and let $V_1,V_2\subseteq V$ be a pair of disjoint subsets of vertices of $G$. The \df{density} of 
$(V_1,V_2)$ is defined as 
$$d(V_1,V_2)=\frac{e(V_1,V_2)}{|V_1||V_2|}\,.$$ 
We say that a pair $(V_1,V_2)$ is \df{$\eps$-regular} for some $\eps>0$ if for every pair of subsets $W_1\subseteq V_1$ and 
$W_2\subseteq V_2$ such that $|W_1|\ge\eps |V_1|$ and $|W_2|\ge \eps |V_2|$ we have 
$$|d(V_1,V_2)-d(W_1,W_2)|\le \eps\,.$$
We call an $\eps$-regular pair $(V_1,V_2)$  \df{strongly $\eps$-regular} if every vertex $v_1\in V_1$ has at least $d(V_1,V_2)|V_2|/10$ neighbors in $V_2$, and every vertex $v_2\in V_2$ has at least  $d(V_1,V_2)|V_1|/10$ neighbors in $V_1$. 

The next result states some elementary results on regular pairs.

\begin{fact}\label{f:RG}
For every  $0<\eps<d/100$ there exists $n'$ such that for every subsets $V_1$, $V_2$ of vertices of 
$G=(V,E)$ such that $|V_1|,|V_2|\ge n'$ the following holds. 
\begin{enumerate}[label=\rmlabel]
\item\label{f5:i}For each $\eps$-regular pair  $(V_1,V_2)$ of density $d$ there are sets $W_1\subseteq V_1$, 
$|W_1|\ge (1-2\eps)|V_1|$, $W_2\subseteq V_2$, 
$|W_2|\ge (1-2\eps)|V_2|$, such that the pair $(W_1,W_2)$ is strongly $\eps$-regular.
\item\label{f5:ii} If $(V_1,V_2)$ is a strongly $\eps$-regular pair and $v_1\in V_1$, $v_2\in V_2$ then for every odd $k$ 
such that $$3\le k\le 2(1-2\eps/d)\min\{|V_1|,|V_2|\}$$
there exists a path of length $k$ which starts at $v_1$ and ends in $v_2$. 
\end{enumerate}
\end{fact}
\begin{proof} Both parts of the above results are well known and direct consequence of the definition. Thus, part \ref{f5:i}
is obtained by removing vertices of small degree from both sets of $\eps$-regular pairs. A path 
required in  \ref{f5:ii} 
can be built greedily, when in each step we choose a vertex which has a lot of neighbors among vertices which do not 
belong to the part of the path we generated so far.
\end{proof}

Finally,  let us  define 
an {\bf $\eps$-regular partition} of a graph $G=(V,E)$ to be a partition of vertices of $G$ into $s$ parts $V=W_1\cup W_2\cup\dots \cup W_s$ so that
\begin{enumerate}[label=\rmlabel]
\item $\big||W_i|-|W_j|\big|\le 1$, for all $1\le i < j \le s$;
\item among $\binom s2$ pairs $(W_i,W_j)$ all but at most $\eps  s^2/2$ are $\eps$-regular. 
\end{enumerate} 

Our argument is based on  the following version of the well known Szemer\'edi's Regularity Lemma.

\begin{lemma}\label{l:SRL}
For every $\eps>0$ there exists $S$ with the following property. For every graph $G=(V,E)$ whose edges were colored with two colors, i.e. $E=\red{R}\cup \blu{B}$, there is a partition $V=W_1\cup W_2\cup\dots \cup W_s$ of vertices of $G$ into $s$ parts, where   $1/\eps\le s\le S$, which is $\eps$-regular for both graphs 
$\red{G_R}=(V,\red{R})$ and $\blu{G_B}=(V,\blu{B})$.
\end{lemma}

The main motivation of using the Regularity Lemma for studying the Ramsey numbers of sparse structures such as cycles is based on the following simple observation due to  {\L}uczak~\cite{L}. Let us suppose that we apply 
Lemma~\ref{l:SRL} to a graph  with some $\eps>0$. We construct a (2-colored) {\bf $\eps$-reduced graph} $\Gepp$ by replacing each subset $W_i$ by a  single vertex $\bw_i$, and if a pair $(W_i,W_j)$ is $\eps$-regular in both colors we color the edge $\bw_i\bw_j$  with the color which appear more frequently among edges joining these sets (in case of a draw we can choose any of the colors, say, the first one). 
In such a way we obtain a 2-coloring of edges of a graph $\kepp$ obtained from the complete graph from which we removed not more than $\eps s^2$ edges. Note that if $G$ is $n$-fit and the edge $\bw_i\bw_j$
of $\Gepp$ is colored with some color, the density of the pair $(W_i,W_j)$ in $\tG$ in this color is at least $1/5$.
It turns out that, instead of looking for a long monochromatic cycle in $G$, it is enough to find large matching 
contained in one monochromatic component in the auxiliary graph $\Gepp$. 

In order to make this statement precise, we say that a graph has property $M_t$ if there exists a matching saturating $\lceil t\rceil$ vertices which is contained in one non-bipartite component of this graph.
It is well known (see {\L}uczak~\cite{L} and
Figaj and {\L}uczak~\cite{FL1, FL2}) that if a graph induced by one of the colors of 2-colored $\eps$-reduced graph 
$\Gepp$ has property $M_t$, then $G$ contains cycles in this color for every length from $3s$ to 
$(t/s-12\eps)2n$. 
Here we need a slightly stronger statement.

\begin{lemma}\label{l:match}
	Let us suppose that an $n$-fit graph $\tG=(V,E)$ is colored with two colors, \red{red} and \blu{blue}, and $\Gepp$ 
	is an $\eps$-reduced graph for this coloring for some $\eps$, where $0<\eps<10^{-5}$. Then, for any constant $a$, $0<a<1$, the following holds.   
If the subgraph induced in $\Gepp$ by one of the colors, say \blu{blue}, has property $M_{as}$, then there exists a set $W\subseteq V$ such that $|W|\ge (a-2\eps)2n$ and for every two vertices $w,w'\in W$  and each  $\ell$, $3s\le \ell\le (a-12\eps)2n$, there exists a monochromatic path of length $\ell$ joining $w,w'$. 
\end{lemma}

\begin{proof}[Sketch of the proof of Lemma~\ref{l:match}] The proof follows closely the argument presented in \cite{L, FL1, FL2} so we just outline it omitting technical details.  
Consider a \blu{blue} component of  $\Gepp$ which contains a matching $\blu{\mathbf{M}}$ saturating 
$\lceil as\rceil$ vertices.  Let $\blu{\mathbf{U}}$ be a unicyclic \blu{blue} subgraph which contains an odd cycle and 
all edges of $\blu{\mathbf{M}}$.  It is easy to see that then there exists a closed walk $\blu{\mathbf{L}}$ which goes through every edge of $\blu{\mathbf{U}}$ twice. Thus, for every two vertices  $\bw'$, $\bw''$  of $\blu{\mathbf{U}}$ there exists 
two walks $\blu{\mathbf{L}_e(\bw', \bw'')},\, \blu{\mathbf{L}_o(\bw', \bw'')}\subseteq \blu{\mathbf{L}}$, of even and odd lengths respectively, 
which start at $\bw'$, end in $\bw''$ and contain each edge of $\blu{\mathbf {M}}$. Note that none of these walks is longer than $2s$. 

Now take all pairs $(W_t, W_u)$ which correspond to the edges in the matching $\blu{\mathbf{M}}$ and delete from each of them 
	some vertices in order to make them strongly $\eps$-regular. The union of all sets which belong to these pairs is our set $W$. Note that, by Fact~\ref{f:RG}\ref{f5:i}, $|W|\ge (a-2\eps)2n$. 

	Next, let us choose $\ell$,  $3s\le \ell\le (a-12\eps)2n$,
	 and take any two vertices $v_i,v_j\in W$ such that $v_i\in W_i$ and  $v_j\in W_j$. Then there exists in  $\Gepp$ a \blu{blue} walk $\blu{\mathbf{L}(\bw_i, \bw_j)}$ of length at most $2s$ and of the same parity as $\ell$ from $\bw_i$  to $\bw_j$ which traverses all the edges from the matching $\blu{\mathbf{M}}$. 
 Now `lift' this walk to a \blu{blue} path $\blu{P}$ in $\tG$ of length at most $2s+4\le 3s$ which joins $v_i$ and $v_j$ and 
	uses at least one edge from each pair $(W_t, W_u)$ which corresponds to an edge of the matching $\blu{\mathbf{M}}$. 
	A slight increase in the length is caused by the fact that the ends of the path $\blu{P}$, vertices $v_i$ and $v_j$, could have few or no neighbors outside the 
	strongly $\eps$-regular pairs they belong to; for instance $v_i\in W_i$ 
could have no neighbors outside the strongly regular pair 
	$(W_i, W_k)$.  In this case we need to start the path $\blu{P}$ with vertices $v_iu_1u_2$, where $u_1\in W_k$,  
	$u_2\in W_i$, where $u_2$ has a lot of neighbors in the next set which correspond to the second vertex of the walk $\blu{\mathbf{L}(\bw_i, \bw_j)}$. Once we have such a path $\blu{P}$ we can increase its length by replacing each edge contained in a strongly $\eps$-regular pairs by  an odd path of required length (see Fact~\ref{f:RG}\ref{f5:ii}). 
\end{proof}

Note that the above results allow us to study properties of colorings of $\kepp$ instead of colorings of an $n$-fit graph $\tG$. Thus, we show that if a coloring of $\tG$ avoids monochromatic cycles 
of required length (which means that in $\Gepp$  monochromatic graphs either have no property $M_t$ 
or contain no triangles), then in $\Gepp$ there exists a large monochromatic bipartite  component (see Lemmata~\ref{l:1} and \ref{l:2} below). In the final step of the proof, we  `lift' a part of this  component to a large monochromatic bipartite graph in  $\tG$.   

We start with a few observations on a graph $\kepp$ which, let us recall,  is obtain from the complete graph $K_s$ by removing at most $\eps s^2$ edges.
\begin{fact}\label{f:degree}
	All but at most $\sqrt{\eps}s$ vertices of $\kepp$ have at least $(1-2\sqrt{\eps})s$ neighbors. 
\end{fact}
\begin{proof}
Indeed, otherwise the number of `missing' edges would be larger than   
		$$(\sqrt{\eps}s)\cdot(2\sqrt{\eps}s)/2=\eps s^2,$$
contradicting the definition of $\kepp$.
\end{proof}

\begin{fact}\label{f:match}
	Let $\eps > 0$, $s\ge 1/\sqrt{\eps}$, and let  $W\subseteq V$, $|W|\ge 11\sqrt{\eps}s$, be any set of vertices of $\kepp$. Then there exists $\widetilde{W}\subseteq W$, $|\widetilde{W}|\ge |W|-2\sqrt{\eps}s$, such that the graph  $\kepp[\widetilde W]$ spanned by $\widetilde W$ in $\kepp$ is hamiltonian and contains a triangle (and so it is a 
	non-bipartite connected graph with a perfect matching). 
\end{fact}
\begin{proof}
	We let $\widetilde{W}$ to be a set of vertices of $V(\kepp)$ obtained from $W$ by removing all vertices with degree in $\kepp$ less than $(1-2\sqrt{\eps})s$ plus, perhaps, one vertex more so that $|\widetilde{W}|$ is even. Then, from Fact~\ref{f:degree}, we get 
	$$|\widetilde{W}|\ge |W|-\sqrt{\eps}s-1 \ge 9\sqrt{\eps}s,$$ and 
$$
			\delta(\kepp[\widetilde{W}]) \ge |\widetilde{W}|-4\sqrt{\eps}s > |\widetilde{W}|/2.
	$$
Hence, from Dirac's Theorem, $\kepp[\widetilde{W}]$ is hamiltonian and the above condition 
also clearly implies that it contains a triangle. 
\end{proof}

\def\wo{{W_1}}
\def\wt{{W_2}}
\def\w{{W}}

Now we can state and prove two main results concerning the coloring of $\kepp$. 

\begin{lemma}\label{l:1}
	For every positive real $\eps\le 10^{-5}$ and integer $s\ge 1/\sqrt{\eps}$ the following holds. 
	If the edges of $\kepp$ are partitioned into two graphs $\blu{B}$ and $\red{R}$ in such a way that \red{$R$} does not have property  $\red{M_{(1/2+13\eps)s}}$ and $\blu{K_3}\not\subseteq \blu{B}$, then  $\blu{B}$ contains an induced  bipartite subgraph with a bipartition $(\wo,\wt)$, such that $|\wo|,|\wt|\ge (1/2-13\sqrt\eps )s$.
\end{lemma}
\begin{proof}
Since $\blu{K_3}\not\subseteq \blu{B}$, from Fact~\ref{f:match} it follows that $\red{R}$ contains no independent sets on $11\sqrt\eps s$ vertices. 
Consequently, none of components of $\red{R}$ with more than 
$22\sqrt\eps s$ vertices is bipartite. 
Hence, our assumption that the \red{red} graph $\red{R}$ does not have property $\red{M_{(1/2+13\eps)s}}$ implies that $\red{R}$ contains no components with more than 
	$$(1/2+13\eps)s+11\sqrt{\eps} s\le (1/2+12\sqrt{\eps})s$$
	vertices. 
	
	Furthermore, from Fact \ref{f:degree} we infer  that at most $\sqrt{\eps}s$
	 vertices of $\kepp$ have less than $(1-2\sqrt\eps)s> 0.9s$
	neighbors. Consequently, because of $\blu{K_3}\nsubseteq\blu{B}$, the set of vertices of 
	degrees larger than $0.9s$  spans in $\red{R}$ only two components on vertex sets $\wo$, $\wt$, 
	$$(1/2+12\sqrt{\eps})s\ge |\wo|\ge |\wt|\ge s-|\wo|-\sqrt{\eps}s\ge (1/2-13\sqrt{\eps})s.$$
	Now the assertion follows from the fact that each pair of vertices from ${W_i}$ has a \blu{blue} common neighbor in ${W_j}$, $\{i,j\}=\{1,2\}$ and thereby, since $\blu{K_3}\not\subseteq \blu{B}$, all the edges contained in either $W_1$ or $W_2$, 
must be  \red{red}.
\end{proof}

\begin{lemma}\label{l:2}
	For every positive real $\eps\le 10^{-5}$ and a large enough integer $s$ the following holds. 
	If the edges of $\kepp$ are partitioned into two graphs $\blu{B}$ and $\red{R}$ in
	such a way that none of them has property  $M_{(1/2+
	13 \eps)s}$, then one of the colors contains an induced  bipartite subgraph with a bipartition $(\wo,\wt)$, such that $|\wo|,|\wt|\ge (1/2-4\sqrt{\eps} )s$.  
\end{lemma}
\begin{proof} Benevides {\it et al.}~\cite{Ben} proved that if $t$ is large enough, then for every graph $G$ 
on $t$ vertices with the minimum degree at least $0.74t$ we have  $G\to (\red{P_{0.66t}}, \blu{P_{0.66t}})$. 
Thus, because of $1-3\sqrt{\eps} > 0.74$,
 from Fact~\ref{f:degree} it follows that in one of the colors, say \red{red}, there exists a monochromatic path $\red{P}$ whose length is at least 
$$0.66(1-\sqrt{\eps})s\ge  0.65 s > (1/2+14\eps)s\,.$$

Since the \red{red} graph $\red{R}$ does not have property $M_{(1/2+
	13 \eps)s}$,  the \red{red} component containing $\red{P}$ must be bipartite with bipartition 
${(W_1, W_2)}$, where $|W_1|\ge |W_2|$. All the edges inside $W_1$ and $W_2$ are \blu{blue}, so our assertion and Fact~\ref{f:match} imply that 
\begin{equation}\label{eq1}   
0.51s\ge (1/2+13{\eps}+2\sqrt{\eps})s \ge |W_1|\ge |W_2|\ge 0.65s- |W_1|\ge  0.14 s\ge 11\sqrt{\eps}s.
\end{equation}
Consequently, using Fact~\ref{f:match} again we infer that the \blu{blue} subgraph induced by $W_1\cup W_2$ 
contains two large cycles $\blu{C_1}$ and $\blu{C_2}$ with at least 
$$|W_1\cup W_2|-4\sqrt{\eps} s\ge 0.65 s -  4\sqrt{\eps} s> (1/2+13\eps)s$$
vertices combined, both contained in non-bipartite \blu{blue} components. Therefore, because $\blu{B}$ does not have property $M_{(1/2+13 \eps)s}$, there is no vertex in $T=V \setminus (W_1\cup W_2)$ with \blu{blue} neighbors in both, $W_1$ and $W_2$. This, in turn, in view of Fact \ref{f:degree} and \eqref{eq1}, implies that $|T|\le \sqrt{\eps}s$, and thereby
$$|W_1|\ge |W_2|= s-|T|- |W_1|\ge  (1/2-4\sqrt\eps )s\,.$$
This completes the proof of Lemma~\ref{l:2}.
\end{proof}

\begin{proof}[Proof of Lemma \ref{l:RL}]
Let us assume that, for some large enough $n_1$, integers $n\ge n_1$, and odd $\ell\le n$ are given. Suppose also that we color the edges of an $n$-fit graph $\tG$ with two colors, \red{red} and \blu{blue}, so that there are no \red{red} copies of $\red{C_n}$ and no \blu{blue} copies of $\blu{C_\ell}$. We argue that there exists two disjoint subsets of vertices $\tG$, $V_1$ and $V_2$, $|V_1|,|V_2|\ge 0.99n$, such that all the edges between $V_1$ and $V_2$ are of the same color. 

To this end we apply the Regularity Lemma to a coloring of $n$-fit graph $\tG$ with some small enough $\eps>0$, so that all arguments hold. Consider an $\eps$-reduced graph $\Gepp$ and denote by $\red{R}$ and $\blu{B}$ its subgraphs consisting of \red{red} and \blu{blue} edges respectively.

We first observe that from our assertion and Lemma \ref{l:match} it follows that 
the \red{red} graph $\red{R}$ does not have property  ${ M_{(1/2+13\eps)s}}$. For the same reason, if $\ell \ge 3s$, the \blu{blue} graph $\blu{B}$ does not have property ${ M_{(1/2+13\eps)s}}$. Moreover,  if $\ell \le 7s$, $\blu{B}$ does not contain a \blu{blue} triangle $\blu{K_3}$. Indeed, note that for even $\ell$, $4\le \ell \le 7s$, the existence of a cycle ${C_\ell}$ in an $\eps$-regular pair $(V_1, V_2)$ is guaranteed by Fact \ref{f:RG}. If in addition also pairs $(V_1,V_3)$ and $(V_2, V_3)$ are $\eps$-regular, again by Fact \ref{f:RG}\ref{f5:i}, we can find in $V_3$ a vertex with many neighbors in both $V_1$ and $V_2$, and thereby, using the $\eps$-regularity of the pair $(V_1, V_2)$, a cycle $C_\ell$ of all odd length $3\le \ell\le 7s$. 

\def\wo{\mathbf{{W_1}}}
\def\wt{\mathbf{{W_2}}}

Thus we can apply Lemma \ref{l:1}, if $\ell \le 7s$, and Lemma \ref{l:2} otherwise, to $\red{R}$ and $\blu{B}$ concluding that one of these graphs contains an induced bipartite subgraph with a bipartition $(\wo,\wt)$, such that $|\wo|,|\wt|\ge (1/2-13\sqrt{\eps})s$. If, in addition,  $\ell \le 7s$, this bipartite subgraph is \blu{blue}. So assume first that $\red{\Gepp[\wo]}, \,\red{\Gepp[\wt]}\subseteq \red{R}$.
 
Now, Fact \ref{f:match} tells us that for both $i=1,2$, there exists $\mathbf{\widetilde{W}_i}\subseteq \mathbf{{W}_i}$, 
	\[
		|\mathbf{\widetilde{W}_i}| \ge |\mathbf{{W}_i}|-2\sqrt{\eps}s \ge 
		(1/2-15\sqrt{\eps})s > 0.498s,
	\]
such that the graph $\red{\Gepp[\mathbf{\widetilde{W}_i}]}$ is connected, non-bipartite, and contains a perfect matching.
In other words, $\red{\Gepp[\mathbf{{W}_i}]}$ has the property ${M}_{0.498s}$ and thereby, in view of Lemma~\ref{l:match}, there exists a set
$V'_i$ of vertices of $\tG$, 
which is contained in the union of the sets corresponding to vertices $\mathbf{w}\in \mathbf{W_i}$, such that 
$
|V'_i|\ge (0.498 - 2\eps)2n\ge 0.995n,
$
and every two vertices of $V'_i$ are connected in $\tG$ by a \red{red} path $\red{P_k}$ of length $k$ for all $k$, 
\[
3s\le k\le0.9n\le (0.495-12\eps)(2n).
\]

Consequently, if there exist two disjoint \red{red} edges in $\tG[V'_1, V_2']$, one can connect their ends by \red{red} paths of lengths $\lfloor \ell/2\rfloor$ and $\lceil \ell/2\rceil-2$, thus obtaining a \red{red} cycle $\red{C_\ell}$, of any length $\ell$, $7s\le \ell \le n$. 
But  $\red{C_n}\nsubseteq \red{R}$ so if we remove from $\tG[V'_1, V_2']$ ends of one \red{red} edge (if there exists any) we remove them all. Thus, we arrive at   $V_1, V_2 \subseteq V$ with $|V_i| \ge |V'_i|-1$ for $i=1,2$ and all edges joining $V_1$ and $V_2$ in \red{red} color.   

In case $\blu{\Gepp[\mathbf{{W}_i}]}\subseteq \blu{B}$ we have $\ell \ge 7s$ and one can repeat the above reasoning with swapped colors.
\end{proof}

\section{Proof of Lemma~\ref{l:D}}\label{sec:D}

Our goal in this section is to establish Lemma~\ref{l:D}. To this end, in what follows, we study an $n$-fit graph\df{ $\tG=(V,E)$} whose edges are colored with two colors, \red{red} and \blu{ blue}, so that  there exist two disjoint sets \df{$V_1,V_2\subseteq V$}, $|V_1|=|V_2|=0.99n$, such that all edges between these two sets are colored with the same color, say \blu{blue}. Denote by $\red{R}$ and $\blu{B}$ subgraphs of $\tG$ induced by \red{red} and \blu{blue} edges respectively. Set \df{$W=V\setminus (V_1\cup V_2)$}. 
Moreover, we assume that  $n$ is large enough so that all 
inequalities below hold.

We start with a few simple observations.
The first one says that most vertices have a lot of neighbors in both sets, $V_1$ and $V_2$. 

\begin{claim}\label{cl:degree}
	All vertices, except at most one, have at least $0.23n$ neighbors in each of the sets $V_1$ and $V_2$. 
\end{claim}
\begin{proof}
	Let us suppose that there exists a vertex $w$ of $\tG$, which has fewer than $0.23n$ neighbors in one of the sets $V_1$, $V_2$, say $V_1$.  Each vertex $v$ shares with $w$ at least $0.5n-n^{0.7} > 0.49n$ neighbors, and so 
		\begin{multline*}
			|N(v)\cap  V_2|\ge |N(v)\cap N(w)|-|N(w)\cap V_1|-|W|
			> 0.49n-0.23n-0.02n>0.23n\,.
		\end{multline*}
	Moreover, 
		\begin{multline*}
			|N(v)\cup N(w)|= |N(v)|+|N(w)|-|N(v)\cap N(w)|
			\ge 1.5n-n^{0.7}> 1.49n,
		\end{multline*}
	so
		\begin{multline*}
			|N(v)\cap V_1|\ge |N(v)\cup N(w)|-|V\setminus V_1|-|N(w)\cap V_1|\\
			\ge 1.49n-1.01n-0.23n>0.23n\,.\qed
		\end{multline*}
	\renewcommand{\qed}{}
\end{proof}

If $\tG$ contains a vertex which has fewer than $0.23n$ neighbors in one of the sets we call it \df{special} and denote by \df{$s$}. In order to simplify the notation we shall always assume that if such special vertex exists, it does not belong to $V_1\cup V_2$. Otherwise we just remove $s$ from $V_1\cup V_2$ and the tiny change of the size of one of the sets $V_1, V_2$ would not affect our estimates, which are always quite crude. 

\def\rp{\red{P_k}}
\def\bp{\blu{P_k}}

Now we show that if there are vertices with a lot of \blu{blue} neighbors in both $V_1$ 
and~$V_2$, we are done.

\begin{claim}\label{cl:bluecycle}
	If there exists a vertex $w\in V$ with at least $n^{0.9}$ \blu{blue} neighbors in both sets $V_1$ and $V_2$, then the \blu{blue} graph $\blu{B}$ contains a cycle \blu{$C_\ell$} for each $\ell$ with $3\le \ell\le n$.
\end{claim}
\begin{proof} Note first that the pair $(V_1,V_2)$ is strongly $\eps$-regular in the \blu{blue} graph for every constant $\eps>0$ (in fact one can take even $\eps=n^{-0.01}$). Thus, by Fact~\ref{f:RG}\ref{f5:ii}, each \blu{blue} edge 
joining vertices $v_1\in V_1$, $v_2\in V_2$, is contained in a cycle of each even length $\ell_e$, where 
$4\le\ell_e\le n$. 
The very same Fact~\ref{f:RG}\ref{f5:ii} implies that 
$w$ belongs to a \blu{blue} cycle of odd length $\ell_o$, provided $5\le \ell_o\le n$.  Finally, by the definition of $n$-fit graphs, 
the sets  $\blu{N_B(w)}\cap V_1$ and $\blu{N_B(w)}\cap V_2$ are joined by more than $n^{1.8}/5$ edges and
all of them are \blu{blue}, hence \blu{$B$} contains a lot of triangles.  
\end{proof}

Therefore, from now on we assume that all vertices have a lot of \blu{blue} neighbors only in one of the sets 
$V_1$, $V_2$. In particular, in view of Claim \ref{cl:degree}, sets $W_1$ and $W_2$, defined for $i=1,2$,
\begin{equation}\label{eq:wi}
	W_i = \{w\in V\colon {|\red{N_R(w)}\cap  V_i|}\ge 0.22n \textrm{ and } {|\blu{N_B(w)}\cap V_i|}\le n^{0.9}\}, 
\end{equation}
cover the whole set $V$ except, perhaps, the special vertex $s$.  Note also that $V_i\subseteq W_i$ for $i=1,2$. 
Our next claim states that the subgraph spanned in the \red{red} graph by each of the sets $W_i$, $i=1,2$, is  pancyclic.
In fact we prove a slightly stronger statements.

\begin{claim}\label{cl:redcycle}
For every $i=1,2$ and each $\ell$, where $3\le \ell\le |W_i|$, the subgraph spanned by $W_i$ in the \red{red} graph contains a cycle of length $\ell$. Moreover, if for some 
$i=1,2$ there exists
a vertex $v\notin W_i$ which has at least two \red{red} neighbors in $W_i$, then
the \red{red} graph contains a cycle of length $|W_i|+1$. 
\end{claim}
\begin{proof}
Let us first find \red{red} cycles of length $\ell\le 0.97n$. To this end, take any vertex $w\in V_i$, and partition  $V_i\setminus \{w\}$ into two roughly equal sets $V_i=W'_i\cup W''_i$ such that each of them contains at least 
$0.1n$ \red{red} neighbors of $w$. Note that for each small constant $\eps>0$ the pair $(W'_i,  W''_i)$ is $\eps$-regular in
$\tG$ and, more importantly, by Claim~\ref{cl:bluecycle}, also in $\red{R}$. Thus, by Fact~\ref{f:RG}\ref{f5:i} it contains a strongly $\eps$-regular pair $(\hat W'_i, 
\hat W''_i)$ in $\red{R}$ of almost the same size. From that we infer, in the same way as we did in Claim~\ref{cl:bluecycle},  that $V_i$ contains cycles of all length $\ell$ up to $|V_i|-30\eps n\ge 0.97n$.

In search of longer cycles we use the well-known argument of Chv\'atal and Erd\H{o}s~\cite{CE}. Let us observe 
first  that  if $v$ has at least two \red{red} neighbors in $W_i$ then it is contained in a \red{red} cycle of length at least $0.96n$. Indeed, then one can split $V_i$ into two equal sets such that one 
contains at least $0.1n$ \red{red} neighbors of one \red{red} neighbor of $s$, the other contains at least 
$0.1n$ \red{red} neighbors of the other \red{red} neighbor of $s$, and repeat verbatim our previous  argument.

Now let us assume that we have already constructed a \red{red} cycle $\red{C_\ell}=\red{v_1v_2\cdots v_\ell v_1}$ for some $\ell\ge 0.96n$, and let $w$ denote a vertex of $W_i$ which does not belong to it. Since $w$ has at least 
$0.22n$ \red{red} neighbors in $V_i$, at least $0.17n$ of them lay on $\red{C_\ell}$. 
We say that a \red{red} neighbor of $w$ in $V(\red{C_\ell})$ is \df{good} if its predecessor in $\red{C_\ell}$ belongs to $V_i$. 
Denote the set of predecessors in \red{$C_\ell$} of all good vertices by $U$. Observe that because of $|W_i-V_i|\le 0.02n $, $|U| \ge 0.1n$ and thereby $\tG[U]$ contains $O(n^2)$ edges. This however, in view of \eqref{eq:wi}, implies that there are two \red{red} good neighbors of $w$, say $v_i$ and $v_j$, $i<j$, whose predecessors are adjacent in 
$\tG$ by a \red{red} edge. 
Then, the cycle  
$$\red{v_1v_2\cdots v_{i-1}v_{j-1}v_{j-2}\cdots v_i w v_jv_{j+1}\cdots v_\ell v_1}$$ 
is a \red{red} cycle of length $\ell+1$ (see Figure \ref{fig:cycle}).
	\begin{figure}[ht]
	\centering
	\begin{tikzpicture}[scale=.5]
		\coordinate (a) at (-.3,1.2);
		\coordinate (b) at (.3, 1.2);
		\coordinate (c) at (2,0);
		\coordinate (d) at (.3,-1.2);
		\coordinate (e) at (-.3, -1.2);
		\coordinate (f) at (-2,0);
		\coordinate (v) at (-.5, 0);
		\draw [red!40, line width=2.5pt] (v)--(b) [out=0, in=90] to (c) [out=-90, in=0] to (d)--(a) [out = 180, in=90] to (f) [out=-90, in =180] to (e)--(v);	
		\draw (a)--(b) [out=0, in=90] to (c) [out=-90, in=0] to (d)--(e) [out = 180, in=-90] to (f) [out=90, in =180] to (a)--(d);	
		\draw (b)--(v)--(e);
		\foreach \i in {a, b, d, e, v, f} \fill (\i) circle (2pt);	
		\node at (-.7, .2) {\tiny $w$};
		\node at (.5, 1.5) {\tiny $v_i$};
		\node at (-.5, -1.6) {\tiny $v_j$};
		\node at (-.5, 1.5) {\tiny $v_{i-1}$};
		\node at (.5, -1.6) {\tiny $v_{j-1}$};
		\node at (-2.3, .2) {\tiny $v_1$};
		\node at (1.3,.4) {\red{ $C_\ell$}};
	\end{tikzpicture}
	\caption{The enlarged cycle.}
	\label{fig:cycle}
\end{figure}
In such a way we can include to the cycle all remaining vertices of $W_i$ 
one by one.  
	\end{proof}

\begin{proof}[Proof of Lemma~\ref{l:D}]
From Claim~\ref{cl:redcycle} it follows that we are done unless $|W_1| = |W_2| = n-1$, $W_1\cap W_2=\emptyset$, and 
the special vertex $s$ exists and has at most one \red{red} neighbor in each of the sets $W_1$ and $W_2$. 
Note that since  $\delta(\tG) = n+1$, $s$ must have in each of the sets $W_1$, $W_2$,  at least one \blu{blue} neighbor. Moreover, by Claim~\ref{cl:bluecycle}, 
in one of these sets, say $W_2$, $s$ has fewer than $n^{0.9}$ \blu{blue} neighbors. Since the degree 
of $s$ is at least $n+1$ so it has at least $n-1-n^{0.9}$ \blu{blue} neighbors in $W_1$ and so  
\begin{equation}\label{eq:3}
|\blu{N_B(s)}\cap V_1| > n-1-n^{0.9}-|W|>0.97n.
\end{equation}
				
Now let $w$ denote a \blu{blue} neighbor of $s$ in $W_2$.  Note that if $w$ has at least two \red{red} neighbors in $V_1$ we are done by Claim~\ref{cl:redcycle}. Thus, 
in view of Claim~\ref{cl:degree}, $w$ must have at least $0.23n-2\ge 0.2n$ \blu{blue} neighbors in 
$V_1$. We show that in this case the graph $\blu{B}$ contains a \blu{blue} cycle \blu{$C_\ell$} of any length $\ell$, $3\le \ell\le n$. 

Let $v_1$ denote and \blu{blue} neighbor of $w$ in $\blu{N_B(s)}\cap V_1$, 
$v_2$ stand for a \blu{blue} neighbor of $v_1$ in $V_2$ such that $v_2\neq w$, and let
$v_3$ be a \blu{blue} neighbor of $v_2$ in $(\blu{N_B(s)}\cap V_1)\setminus \{v_1\}$.
 Note that from (\ref{eq:3}) and Claim~\ref{cl:degree} it follows that such vertices $v_1,v_2,v_3$ always exist. Then $\blu{swv_1s}$ and $\blu{swv_1v_2v_3s}$ are \blu{blue} cycles of length 3 and $5$, respectively. Moreover, the pair 
$(V_1\setminus \{v_1\}, V_2\setminus \{w\})$ is strongly $\eps$-regular for each constant $\eps>0$, so Fact~\ref{f:RG}\ref{f5:ii} implies that for every odd length $\ell_o$, $7\le \ell_o\le n+4$ the path $\blu{v_3swv_1v_2}$ is contained in an odd \blu{blue} cycle of length $\ell_o$ and the \blu{blue} edge $\blu{v_2v_3}$ belongs to an even \blu{blue} cycle 
of length $\ell_e$ for each $4\le \ell_e\le n$. Thus, the \blu{blue} graph $\blu{B}$ contains cycles of every length between $3$ and $n$.
This completes the proof of Lemma~\ref{l:D}.
	\begin{figure}[ht]
	\centering
	\begin{tikzpicture}[scale=.68]
		\coordinate (s) at (3, 0);
		\coordinate (w) at (2.35,-1.2);
		\coordinate (a) at (.5,1.35);
		\coordinate (c) at (.1, -1.3);
		\coordinate (d) at (-.5, 1);		
		\path [draw=red!80!black] (0,-1.3) ellipse (2.5cm and .75cm);
		\path [draw=red!80!black] (0,1.3) ellipse (2.5cm and .75cm);
		\draw [red!60!black, thick] (-.1,1.3) ellipse (2.3cm and .6cm);
		\fill [red!10] (-.1,1.3) ellipse (2.3cm and .6cm);
		\draw [red!60!black, thick] (-.1,-1.3) ellipse (2.3cm and .6cm);
		\fill [red!10] (-.1,-1.3) ellipse (2.3cm and .6cm);
		\draw [blue!80!black] (s)--(w)--(a);
		\draw [ultra thick, decorate, decoration={snake}, blue!80!black] (d)[out = 180, in= 90] to (-1,0) [out=-90, in=180] to (c);
		\draw [blue!80!black, ultra thick] (d)--(c);
			\draw [blue!80!black, ultra thick]  (a)--(s);
				\draw [blue!80!black, ultra thick] (s)--cycle;
		\draw [blue!80!black,ultra thick] (d)--(s);
			\draw [blue!80!black, ultra thick] (s)--(w);
			\draw [blue!80!black, ultra thick] (w)--(a);
				\draw [blue!80!black, ultra thick] (s)--(a);
					\draw [blue!80!black, ultra thick] (a)--(c); 
		\draw [blue!80!black,ultra thick] (d)--(c); 		 
		\foreach \i in {s,w,a, c,d} \fill (\i) circle (2pt);	 
		\node [right] at (s) {{\small $s$}};
		\node [right] at (w) {\small$w$};
		\node at (.6,1.6) {\small$v_1$};
		\node [below] at (c) {\small$v_2$};
		\node [above] at (d) {\small$v_3$};
		\node at (-3,-.1) {\blu{$P_{\ell_e}=P_{\ell_o-3}$}};
		\node at (3.4,-1.7) {\large \red{$W_2$}};
		\node at (3.4,1.5) {\large \red{$W_1$}};
		\node at (-1.7,-1.3) {\large \red{$V_2$}};
		\node at (-1.7,1.3) {\large \red{$V_1$}};	
	\end{tikzpicture}
	\caption{Constructing  \blu{blue cycles}. Short odd cycles $\blu{swv_1s}$ and $\blu{swv_1v_2v_3s}$ are clearly visible, while all odd cycles of length at least 7, as well as all even cycles, are constructed using wavy paths of appropriate lengths. }
	\label{fig:Cnb2}
\end{figure}
\end{proof}


\begin{bibdiv}
	\begin{biblist}

		\bib{A}{article}{
			  title={Ramsey numbers of cycles in random graphs},
			  author={Ara\'ujo, Pedro},
			  author={Pavez-Sign\'e, Mat\'ias},
			  author = {Sanhueza-Matamala, Nicol\'as}
			  	note={Submitted},
		} 
		
		
		
		\bib{Ben}{article}{
   author={Benevides, F. S.},
   author={\L uczak, T.},
   author={Scott, A.},
   author={Skokan, J.},
   author={White, M.},
   title={Monochromatic cycles in 2-coloured graphs},
   journal={Combin. Probab. Comput.},
   volume={21},
   date={2012},
   number={1-2},
   pages={57--87},
   review={\MR{2900048}},
}
			
		\bib{BD}{article}{
			  title={New lower bounds on the size-Ramsey number of a path},
			  author={Bal, Deepak},
			  author={DeBiasio, Louis},
			  journal={The Electronic Journal of Combinatorics},
			  year={2022}
			  	volume={29},
			  	number={1},
			  	pages={\#P1.18},
			  	review={\MR{4396465}},
		} 
		
		
		
		\bib{B}{article}{
			title={On size Ramsey number of paths, trees, and circuits. I},
			author={Beck, J{\'o}zsef},
			journal={Journal of Graph Theory},
			volume={7},
			number={1},
			pages={115--129},
			date={1983},
			publisher={Wiley Online Library},
			review={\MR{0693028}}
		}
		
		\bib{CE}{article}{
   author={Chv\'{a}tal, V.},
   author={Erd\H{o}s, P.},
   title={A note on Hamiltonian circuits},
   journal={Discrete Math.},
   volume={2},
   date={1972},
   pages={111--113},
   issn={0012-365X},
   review={\MR{297600}},
   doi={10.1016/0012-365X(72)90079-9},
}
		
		\bib{DP}{article}{
			  title={On some multicolor Ramsey properties of random graphs},
			  author={Dudek, Andrzej},
			  author= {Pra{\l}at, Pawe{\l}},
			  journal={SIAM Journal on Discrete Mathematics},
			  volume={31},
			  number={3},
			  pages={2079--2092},
			  year={2017},
			  publisher={SIAM},
			  review={\MR{3697158}},
			} 
			
		\bib{FS}{article}{
			title={All Ramsey numbers for cycles in graphs},
			author={Faudree, Ralph J},
			 author={Schelp, Richard H},
			journal={Discrete Mathematics},
			volume={8},
			number={4},
			pages={313--329},
			year={1974},
			publisher={Elsevier},
			review={\MR{0345866}},
		} 
	
	\bib{FL1} {article}{
		title={The Ramsey number for a triple of long even cycles},
		author={Figaj, Agnieszka},
		author = { {\L}uczak, Tomasz},
		journal={Journal of Combinatorial Theory, Series B},
		volume={97},
		number={4},
		pages={584--596},
		year={2007},
		publisher={Elsevier}
	}
	
	\bib{FL2} {article}{
		title={The Ramsey numbers for a triple of long cycles},
		author={Figaj, Agnieszka},
		author= {{\L}uczak, Tomasz},
		journal={Combinatorica},
		volume={38},
		number={4},
		pages={827--845},
		year={2018}
	}



\bib{FS}{article}{
   author={Faudree, R. J.},
   author={Schelp, R. H.},
   title={A survey of results on the size Ramsey number},
   conference={
      title={Paul Erd\H{o}s and his mathematics, II},
      address={Budapest},
      date={1999},
   },
   book={
      series={Bolyai Soc. Math. Stud.},
      volume={11},
      publisher={J\'{a}nos Bolyai Math. Soc., Budapest},
   },
   date={2002},
   pages={291--309},
   review={\MR{1954730}},
}
		
		
		\bib{GG}{article}{
			title={On Ramsey-type problems},
			author={Gerencs{\'e}r, L{\'a}szl{\'o}}, 
			author = {Gy{\'a}rf{\'a}s, Andr{\'a}s},
			journal={Ann. Univ. Sci. Budapest. E{\"o}tv{\"o}s Sect. Math},
			volume={10},
			pages={167--170},
			year={1967},
			review={\MR{0239997}},
		} 
		
		\bib{HKL}{article}{
			title={The induced size-Ramsey number of cycles},
			author={Haxell, Penny},
			author={Kohayakawa, Yoshiharu},
			author= {{\L}uczak, Tomasz},
			journal={Combinatorics, Probability and Computing},
			volume={4},
			number={3},
			pages={217--239},
			year={1995},
			publisher={Cambridge University Press},
			review={\MR{1356576}},
		}
		
		\bib{JKOP}{article}{
			title={On the size-Ramsey number of cycles},
			author={Javadi, Ramin},
			author={Khoeini, Farideh},
			author = {Omidi, Gholam Reza}
			author = {Pokrovskiy, Alexey},
			journal={Combinatorics, Probability and Computing},
			volume={28},
			number={6},
			pages={871--880},
			date={2019},
			publisher={Cambridge University Press},
			review={\MR{4015660}}
		}
		
		\bib{KS}{article}{
   author={Koml\'{o}s, J.},
   author={Simonovits, M.},
   title={Szemer\'{e}di's regularity lemma and its applications in graph theory},
   conference={
      title={Combinatorics, Paul Erd\H{o}s is eighty, Vol. 2},
      address={Keszthely},
      date={1993},
   },
   book={
      series={Bolyai Soc. Math. Stud.},
      volume={2},
      publisher={J\'{a}nos Bolyai Math. Soc., Budapest},
   },
   date={1996},
   pages={295--352},
   review={\MR{1395865}},
}
		
		
		\bib{JM}{article}{
			title={Multicolor Size-Ramsey Number of Cycles},
			author={Javadi, Ramin},
			author= {Miralaei, Meysam},
			journal={arXiv preprint arXiv:2106.16023},
			year={2021}
			} 
		\bib{LW}{article}{
			title={Asymptotic enumeration of graphs by degree sequence, and the degree sequence of a random graph},
			author={Liebenau, Anita},
			author={Wormald, Nick},
			journal={arXiv preprint arXiv:1702.08373},
			year={2017}
			} 
		
		
		\bib{Let}{article}{
   author={Letzter, Shoham},
   title={Path Ramsey number for random graphs},
   journal={Combin. Probab. Comput.},
   volume={25},
   date={2016},
   number={4},
   pages={612--622},
   issn={0963-5483},
   review={\MR{3506430}},
   doi={10.1017/S0963548315000279},
}
		
		\bib {L}{article}{
			title={$R(C_n,C_n,C_n)\le (4+o(1))n$},
			author={{\L}uczak, Tomasz},
			journal={Journal of Combinatorial Theory, Series B},
			volume={75},
			number={2},
			pages={174--187},
			year={1999},
			publisher={Elsevier}
		}
		
	\bib{LZ1}{article}{
   author={\L uczak, Tomasz},
   author={Rahimi, Zahra},
   title={On Schelp's problem for three odd long cycles},
   journal={J. Combin. Theory Ser. B},
   volume={143},
   date={2020},
   pages={1--15},
   issn={0095-8956},
   review={\MR{4089572}},
   doi={10.1016/j.jctb.2019.11.002},
}	

\bib{LZ2}{article}{
   author={\L uczak, Tomasz},
   author={Rahimi, Zahra},
   title={Long monochromatic even cycles in 3-edge-coloured graphs of large
   minimum degree},
   journal={J. Graph Theory},
   volume={99},
   date={2022},
   number={4},
   pages={691--714},
   issn={0364-9024},
   review={\MR{4429175}},
   doi={10.1002/jgt.22760},
}
		
		
		\bib{Ra}{article}{
			title={Small ramsey numbers},
			author={Radziszowski, Stanislaw},
			journal={The Electronic Journal of Combinatorics},
			volume={1000},
			pages={DS1--Jan},
			year={2021},
			review={\MR{1670625}}
			} 
		
		\bib{R}{article}{
			title={On a Ramsey-type problem of JA Bondy and P. Erd{\"o}s. I},
			author={Rosta, Vera},
			journal={Journal of Combinatorial Theory, Series B},
			volume={15},
			number={1},
			pages={94--104},
			year={1973},
			publisher={Elsevier},
			review={\MR{0332567}}
		}
		
	\end{biblist}
\end{bibdiv}
\end{document}